\documentclass[12pt,reqno]{amsart}
\usepackage[dvips]{graphicx}
\usepackage{color}

\usepackage{amsmath}
\usepackage{amssymb}
\usepackage{amsthm}
\usepackage{bm}
\numberwithin{equation}{section}
\newtheorem{theorem}{Theorem}[section]
\newtheorem{lemma}[theorem]{Lemma}
\newtheorem{proposition}[theorem]{Proposition}
\newtheorem{corollary}[theorem]{Corollary}

\newtheorem{thm}{Theorem}

\theoremstyle{definition}
\newtheorem*{remark}{Remark}
\newtheorem*{conj}{Conjecture}

\allowdisplaybreaks
\begin{document}
\title[Lerch's formula and zeros of the quadrilateral zeta function]{On Lerch's formula and zeros of the \\quadrilateral zeta function}
\author[T.~Nakamura]{Takashi Nakamura}
\address[T.~Nakamura]{Department of Liberal Arts, Faculty of Science and Technology, Tokyo University of Science, 2641 Yamazaki, Noda-shi, Chiba-ken, 278-8510, Japan}
\email{nakamuratakashi@rs.tus.ac.jp}
\urladdr{https://sites.google.com/site/takashinakamurazeta/}

\maketitle

\begin{abstract}
Let $0 < a \le 1/2$ and define the quadrilateral zeta function by $2Q(s,a) := \zeta (s,a) + \zeta (s,1-a) + {\rm{Li}}_s (e^{2\pi ia}) + {\rm{Li}}_s(e^{2\pi i(1-a)})$, where $\zeta (s,a)$ is the Hurwitz zeta function and ${\rm{Li}}_s (e^{2\pi ia})$ is the periodic zeta function. 

In the present paper, we show that there exists a unique real number $a_0 \in (0,1/2)$ such that $Q(\sigma, a_0)$ has a unique double real zero at $\sigma = 1/2$ when $\sigma \in (0,1)$, for any $a \in (a_0,1/2]$, the function $Q(\sigma, a)$ has no zero in the open interval $\sigma \in (0,1)$ and for any $a \in (0,a_0)$, the function $Q(\sigma, a)$ has at least two real zeros in $\sigma \in (0,1)$.

Moreover, we prove that $Q(s,a)$ has infinitely many complex zeros in the region of absolute convergence and the critical strip when $a \in {\mathbb{Q}} \cap (0,1/2) \setminus \{1/6, 1/4, 1/3\}$. The Lerch formula, Hadamard product formula, Riemann-von Mangoldt formula for $Q(s,a)$ are also shown. 
\end{abstract}

\section{Introduction, Main results and Some remarks}
\subsection{Introduction}
For $0< a \le 1$, we define the Hurwitz zeta function $\zeta (s,a)$ and the periodic zeta function ${\rm{Li}}_s (e^{2\pi ia})$ by 
$$
\zeta (s,a) := \sum_{n=0}^\infty \frac{1}{(n+a)^s}, \quad
{\rm{Li}}_s (e^{2\pi ia}) := \sum_{n=1}^\infty \frac{e^{2\pi ina}}{n^s}, \qquad \sigma >1,
$$
respectively. The both Dirichlet series of $\zeta (s,a)$ and ${\rm{Li}}_s (e^{2\pi ia})$ converge absolutely in the half-plane $\sigma >1$ and uniformly in each compact subset of this region. The Hurwitz zeta function $\zeta (s,a)$ can be extended for all $s \in {\mathbb{C}}$ except $s=1$, where there is a simple pole with residue $1$ (see for instance \cite[Section 12]{Apo}). On the other hand, the Dirichlet series of the function ${\rm{Li}}_s (e^{2\pi ia})$ with $0<a<1$ converges uniformly in each compact subset of the half-plane $\sigma >0$ (see for example \cite[p.~20]{LauGa}). The function ${\rm{Li}}_s (e^{2\pi ia})$ with $0<a<1$ is analytically continuable to the whole complex plane (see for instance \cite[Section 2.2]{LauGa}). 

Obviously, one has $\zeta (s,1) = {\rm{Li}}_s (1) = \zeta (s)$, where $\zeta (s)$ is the Riemann zeta function. We can easily see that $\zeta (\sigma) >0$ when $\sigma >1$ form the series expression of $\zeta (s)$. Since (see for instance \cite[(2.12.4)]{Tit})
$$
\bigl( 1 - 2^{1-\sigma} \bigr) \zeta (\sigma) =
1 - \frac{1}{2^\sigma} + \frac{1}{3^\sigma} - \frac{1}{4^\sigma} + \cdots >0, \qquad 0<\sigma <1,
$$
one has $\zeta (\sigma) <0$ if $0<\sigma <1$. Moreover, we have $\zeta (0) = -1/2 <0$ (see for example \cite[(2.4.3)]{Tit}). Therefore, from the functional equation of $\zeta (s)$, we have the following (see for example \cite[Theorem 1.6.1]{KV} or \cite[Section 2.12]{Tit}).
\begin{thm}
All real zeros of $\zeta(s)$ are {\sc{simple}} and at {\sc{only}} the negative even integers.
\end{thm}

The following Lerch's formula is well-know (see for example \cite[Theorem 1.3.4]{AAR})
\begin{equation}\label{eq:lerch1}
\zeta' (0,a) := \Bigl[ \frac{d}{ds} \zeta (s,a) \Bigr]_{s=0} = \log \frac{\Gamma (a)}{\sqrt{2\pi}}. 
\end{equation}
Especially, the value $\zeta ' (0,1) = -(\log 2\pi )/2$ plays a crucial role in the theory of the Riemann zeta function. For example, the approximation of $\zeta (s)$ near $s=1$ is given by
\begin{equation}\label{eq:kro1RZ}
\biggl[ \zeta(s) -\frac{1}{s-1} \biggr]_{s=1} - 2\zeta '(0) = \gamma_E + \log 2\pi,
\end{equation}
where $\gamma_E $ is the Euler constant (see \cite[(2.1.16)]{Tit}). Furthermore, the value $\zeta ' (0)$ appears in the Hadamard product formula for the completed Riemann zeta function $\xi (s)$. Put $2\xi (s) := s(s-1) \pi^{-s/2} \Gamma (s/2) \zeta (s)$. Then we have
\[
\xi (s,a) = e^{A+Bs} \prod_{\rho} \Bigl( 1 - \frac{s}{\rho} \Bigr) e^{s/\rho},
\]
where the product is over the zeros $\rho$ of $\zeta (s)$, the constants $A$ and $B$ are given as
\begin{equation}\label{eq:hadpro1RZ}
e^A  = \frac{1}{2}, \quad \mbox{ and } \quad B = \frac{\zeta' (0)}{\zeta (0)} - 1 - \frac{\gamma_E+\log \pi}{2}
\end{equation}
(see \cite[Chapter 2.12]{Tit}). Moreover, we have the following Riemann-von Mangoldt formula for $\zeta (s)$ (see \cite[Theorem 1.8.1]{KV} and \cite[Theorem 9.4]{Tit}).  Let $N(T, F)$ count the number of non-real zeros of a function $F(s)$ having $|\Im (s)| < T$. Then we have 
\begin{equation*}\label{eq:RvM1}
N\bigl( T, \zeta(s) \bigr) = \frac{T}{\pi} \log \frac{T}{2\pi} - \frac{T}{\pi} + O(\log T).
\end{equation*}

\subsection{Main results}
For $ 0 <a \le 1/2$, define the quadrilateral zeta function $Q(s,a)$ by
$$
2Q(s,a) := \zeta (s,a) + \zeta (s,1-a) +{\rm{Li}}_s (e^{2\pi ia}) + {\rm{Li}}_s (e^{2\pi i(1-a)}).
$$
The function $Q(s,a)$ can be continued analytically to the whole complex plane except $s=1$. In \cite[Theorem 1.1]{Na}, the author prove the functional equation 
\begin{equation}\label{eq:zfe1q}
Q(1-s,a) = \Gamma_{\!\! \cos} (s)  Q(s,a), \qquad \Gamma_{\!\! \cos} (s) := \frac{2\Gamma (s)}{(2\pi )^s} \cos \Bigl( \frac{\pi s}{2} \Bigr).
\end{equation}
We remark that the gamma factor in (\ref{eq:zfe1q}) completely coincides with that of the functional equation for $\zeta (s)$ (see \cite[Section 1.3]{Na}). Let $N_{\rm{Q}}^{\rm{CL}} (T)$ the number of the zeros of $Q(s,a)$ on the line segment from $1/2$ to $1/2 +iT$. In \cite[Theorem 1.2]{Na}, he showed that for any $0 < a \le 1/2$, there exist positive constants $A(a)$ and $T_0(a)$ such that 
\begin{equation*}\label{eq:ZCLQ1}
N_{\rm{Q}}^{\rm{CL}} (T) \ge A(a) T \quad \mbox{whenever} \quad T \ge T_0(a). 
\end{equation*}

In the present paper, we show the following statements on Lerch's formula, the Hadamard product formula, Riemann-von Mangoldt formula and real and complex zeros of $Q(s,a)$. First, we state results on real zeros of $Q(s,a)$. By Mathematica 11.3, the real number $a_0 \in (0,1/2)$ appeared in the next theorem on real zeros of $Q(s,a)$ is 
\begin{equation}\label{eq:av1}
\begin{split}
a_0 = 0.&11837513961527229358271903455211912971471769999053\\
&14554591427859384268411483278906208314018589873082...
\end{split}
\end{equation}

\begin{theorem}\label{th:m1}
There exists the unique real number $a_0 \in (0,1/2)$ such that \\
{\rm{(1)}} the function $Q(\sigma, a_0)$ has a unique double real zero at $\sigma = 1/2$ when $\sigma \in (0,1)$,\\ 
{\rm{(2)}} for any $a \in (a_0,1/2]$, the function $Q(\sigma, a)$ has no real zero in $\sigma \in (0,1)$,\\
{\rm{(3)}} for any $a \in (0,a_0)$, the function $Q(\sigma, a)$ has at least two real zeros in $\sigma \in (0,1)$. 
\end{theorem}
By this theorem, we have the following as an analogue of Theorem A.
\begin{corollary}\label{cor:1}
All real zeros of the quadrilateral zeta function $Q(s,a)$ are {\sc{simple}} and are located {\sc{only}} at the negative even integers just like $\zeta (s)$ if and only if $a_0 < a \le 1/2$.
\end{corollary}

For non-real zeros of $Q(s,a)$, we have the following propositions. We first consider the cases $a= 1/6$, $1/4$, $1/3$ or $1/2$.
\begin{proposition}\label{pro:1}
Suppose that $a= 1/6$, $1/4$, $1/3$ or $1/2$. Then the Riemann hypothesis is true if and only if all non-real zeros of $Q(s,a)$ are on the critical line $\sigma = 1/2$.
\end{proposition}
On the other hand, we can see that $Q(s,a)$ with $a \in {\mathbb{Q}} \cap (0,1/2) \setminus \{1/6, 1/4, 1/3\}$ does not satisfy an analogue of the Riemann hypothesis. 
\begin{proposition}\label{pro:zero1}
Let $a \in {\mathbb{Q}} \cap (0,1/2) \setminus \{1/6, 1/4, 1/3\}$. Then for any $\delta >0$, there exist positive constants $C_a^\flat (\delta)$ and $C_a^\sharp (\delta)$ such that the function $Q(s,a)$ has more than $C_a^\flat (\delta) T$ and less than $C_a^\sharp (\delta)T$ complex zeros in the rectangles $1 < \sigma < 1+\delta$ and $0 < t < T$, and $-\delta < \sigma < 0$ and $0 < t < T$ if $T$ is sufficiently large. Furthermore, for any $1/2 < \sigma_1 < \sigma_2 <1$, there are positive numbers $C_a^\flat (\sigma_1,\sigma_2)$ and $C_a^\sharp (\sigma_1,\sigma_2)$ such that $Q(s,a)$ has more than $C_a^\flat (\sigma_1,\sigma_2) T$ and less than $C_a^\sharp (\sigma_1,\sigma_2) T$ non-trivial zeros in the rectangles $\sigma_1 < \sigma < \sigma_2$ and $0 < t < T$, $1-\sigma_2 < \sigma < 1-\sigma_1$ and $0 < t < T$ when $T$ is sufficiently large.
\end{proposition}

Next we consider Lerch's formula for $Q(s,a)$ as a analogue of (\ref{eq:lerch1}). We remark that we have $Q(0,a) = -1/2 = \zeta (0)$ by (\ref{eq:Q-1/2}). 
\begin{theorem}\label{th:lerch1}
Let $\psi (a)$ be the digamma function. Then one has
\begin{equation}\label{eq:thmlerch1}
Q' (0,a) = \frac{1}{4} \Bigl( - 2 \log (\sin \pi a) - 2 \log 4\pi - 2\gamma_E -\psi (a) -\psi (1-a) \Bigr). 
\end{equation}
\end{theorem}
We have the following corollary. The equation (\ref{eq:corlerch1}) below gives the approximation of $Q(s,a)$ near $s=1$. It should be noted that the right hand side of (\ref{eq:kro1RZ}) and (\ref{eq:kro1}) below completely coincide. The second statement is an asymptotic expansion of $Q'(0,a)$ as $a \to +0$. The equation (\ref{eq:corlerch1}) implies that $Q' (0,a)$ is written by only the logarithm and trigonometric functions when $0 < a \le 1/2$ is rational.
\begin{corollary}\label{cor:lerch1}
We have
\begin{equation}\label{eq:kro1}
\biggl[ Q(s,a) -\frac{1}{s-1} \biggr]_{s=1} - 2 Q' (0,a) = \gamma_E + \log 2\pi.
\end{equation}
One has $Q' (0,a) \to \infty$ when $a \to +0$. More precisely,
\begin{equation}\label{eq:corasym1}
Q' (0,a) = \frac{1}{4} \biggl( \frac{1}{a} - 2 \log a \biggr) + O(1), \qquad a \to +0.
\end{equation}
Let $r$ and $q$ be relatively prime natural numbers and $0 < r/q \le 1/2$. Then we have
\begin{equation}\label{eq:corlerch1}
Q' (0,r/q) = \frac{1}{2} \biggl( \log \frac{q}{4\pi} -  \log \Bigl( 2\sin \frac{\pi r}{q} \Bigr) - 
\sum_{n=1}^{q-1} \cos \frac{2\pi rn}{q} \log \Bigl( 2 \sin \frac{\pi n}{q} \Bigr) \biggr) .
\end{equation}
\end{corollary}

For $0 <a \le 1/2$, define the function $\xi_Q(s,a)$ by
\begin{equation*}\label{eq:ComQ1}
2\xi_Q(s,a) := s(s-1) \pi^{-s/2} \Gamma (s/2) Q(s,a) .
\end{equation*}
The value $Q' (0,a) $ plays important role in the following Hadamard product formula for $\xi_Q(s,a)$. Note that the constant $B$ in (\ref{eq:hadpro1RZ}) and $B(a)$ given by (\ref{eq:hadpro1Q}) completely coincide if we replace $\zeta (s)$ by $Q(s,a)$. Moreover, we can show an approximate formula for $Q' (s,a) / Q(s,a)$ in terms of zeros of $Q(s,a)$ near to $s$ in Proposition \ref{cor:corHadpro1} as a consequence of the Hadamard product formula for $\xi_Q(s,a)$.
\begin{proposition}\label{pro:Hadpro1}
Let $0 < a \le 1/2$ and $\rho_a$ be the zeros of $\xi_Q (s,a)$. Then we have
\begin{equation*}\label{eq:hadpro1}
\xi_Q (s,a) = e^{A+B(a)s} \prod_{\rho_a} \Bigl( 1 - \frac{s}{\rho_a} \Bigr) e^{s/\rho_a},
\end{equation*}
where $A$ and $B(a)$ are defined as
\begin{equation}\label{eq:hadpro1Q}
e^A  = \frac{1}{2}, \quad \mbox{ and } \quad B(a) = \frac{Q' (0,a)}{Q (0,a)} - 1 - \frac{\gamma_E+\log \pi}{2}.
\end{equation}
\end{proposition}

Furthermore, we have the following Riemann-von Mangoldt formula for $Q(s,a)$.
\begin{proposition}\label{th:czaQAB1}
For $0 <a \le 1/2$ and $T>2$, we have
$$
N\bigl( T, Q(s,a)\bigr) = \frac{T}{\pi} \log \frac{T}{2\pi} - \frac{T}{\pi} - \frac{2 T}{\pi} \log a + O_a(\log T).
$$
\end{proposition}

In the next subsection, we give some remarks for Theorem \ref{th:m1}, Corollary \ref{cor:1} and \ref{cor:lerch1}, Propositions \ref{pro:zero1} and \ref{th:czaQAB1}, and the Epstein zeta function $\zeta_B(s)$ in order to explain that the properties of $Q(s,a)$ are ``better'' than those of $\zeta_B(s)$ .

\subsection{The Epstein zeta function }
As we mention below, the quadrilateral zeta function $Q(s,a)$ has many analytical properties in common with the Epstein zeta function $\zeta_B(s)$ defined below, for example, the functional equation (\ref{eq:feEp}), the Riemann-von Mangoldt formula (\ref{eq:RvMEp1}) and zeros on the critical line (\cite[Theorem 1.2]{Na}). 
 
Let $B(x,y)= a x^2 + bxy + c y^2$ be a positive definite integral binary quadratic form, and denote by $r_{\!B}(n)$ the number of solutions of the equation $B(x,y) = n$ in integers $x$ and $y$. Then the Epstein zeta function for the binary quadratic form $B$ is defined by the series
\begin{equation*}\label{eq:defEp}
\zeta_B(s) := \sum_{(x,y) \in {\mathbb{Z}}^2\setminus (0,0)} \frac{1}{B(x,y)^s}
= \sum_{n=1}^\infty \frac{r_{\!B}(n)}{n^s}
\end{equation*}
for $\sigma >1$. It is widely know that the function $\zeta_B(s)$ admits analytic continuation into the entire complex plane
except for a simple pole at s = 1 with residue $2\pi (-D)^{-1/2}$, where $D := b^2-4ac<0$. Moreover, the function $\zeta_B(s)$ fulfills the functional equation
\begin{equation}\label{eq:feEp}
\Bigl( \frac{\sqrt{-D}}{2\pi} \Bigr)^s \Gamma (s) \zeta_B(s) = 
\Bigl( \frac{\sqrt{-D}}{2\pi} \Bigr)^{1-s} \Gamma (1-s) \zeta_B(1-s) .
\end{equation}

Denoted by $N_{\rm{Ep}}^{\rm{CL}} (T)$ the number of the zeros of the Epstein zeta function $\zeta_B (s)$ on the critical line and whose imaginary part is smaller than $T>0$. Potter and Titchmarsh \cite{PT} showed $N_{\rm{Ep}}^{\rm{CL}} (T) \gg T^{1/2-\varepsilon}$. The current (June, 2022) best result is $N_{\rm{Ep}}^{\rm{CL}} (T) \gg T^{4/7-\varepsilon}$, shown by Baier, Srinivas and Sangale \cite{BSS}.

The distribution of zeros of $\zeta_B(s)$ off the critical line depends on the value of the class number $h(D)$ of the imaginary quadratic field  ${\mathbb{Q}} (\sqrt{D})$. If $h(D)=1$, it is expected that $\zeta_B(s)$ satisfies an analogue of the Riemann hypothesis since $\zeta_B(s)$ has an Euler product. When $h(D) \ge 2$, Davenport and Heilbronn \cite{DH} proved that $\zeta_B(s)$ has infinitely many zeros in the region of absolute convergence $\Re(s) > 1$. Voronin \cite{Vo} proved that if $h(D) \ge 2$, then there exists a positive constant $C_B^\flat (\sigma_1,\sigma_2) $ such that the function $\zeta_B(s)$ has more than $C_B^\flat (\sigma_1,\sigma_2) T$ non-trivial zeros in the rectangle $\sigma_1 < \sigma < \sigma_2$ and $0 < t < T$, when $T$ is sufficiently large (see also \cite[Theorem 7.4.3]{KV}). Recently, for example, Lee \cite{Lee} and Lamzouri \cite{Lam} improve Voronin's theorem. 

Hereafter, let $B(x,y)= a x^2 + bxy + c y^2$ be a positive definite binary quadratic form, namely, $a,b,c \in {\mathbb{R}}$, $a>0$ and $d:= b^2-4ac < 0$. And define $\kappa >0$ by putting
$$
\kappa^2 := \frac{|d|}{4a^2} = \frac{4ac-b^2}{4a^2} = \frac{c}{a} - \Bigl( \frac{b}{2a} \Bigr)^2. 
$$
Bateman and Grosswald \cite{BG} showed the following theorem. It should be noted that this result was announced by Chowla and Selberg \cite{CS} without a proof. 
\begin{thm}
Let $\kappa \ge \sqrt{3}/2$. Then one has
$$
\zeta_B(1/2) > 0 \quad \mbox{if} \quad \kappa \ge 7.00556 
$$
(or if $4\kappa^2 = |d|/a^2 \ge 199.2$) but,
$$
\zeta_B(1/2) < 0 \quad \mbox{if} \quad \sqrt{3}/2 \le \kappa \le 7.00554
$$
(or if $3 \le 4\kappa^2 = |d|/a^2 \le 199.1$). 
\end{thm}
There are no result for the sign of `the central value' $\zeta_B(1/2)$ when $7.00554 < \kappa < 7.00556$ at present. It should be mentioned that $\zeta_B(s)$ vanishes in the interval $(1/2,1)$ if $\kappa \ge 7.00556$ by the theorem above, the intermediate value theorem and $\lim_{\sigma \to 1-0} \zeta_B(\sigma) = -\infty$. Moreover, it is probable that $\zeta_B(\sigma) < 0$ for all $\sigma \in (0,1)$ if $\sqrt{3}/2 \le \kappa \le 7.00554$. 

Furthermore, we have the following Riemann-von Mangoldt formula for $\zeta_B(s)$ (see for example \cite[p.~692]{Steu06}). Denote by ${\rm{m}}(B)$ the minimum of the values of the quadratic form $B(x,y)$ for $(x,y) \ne (0,0)$. Then we have
\begin{equation}\label{eq:RvMEp1}
N\bigl( T, \zeta_B(s) \bigr) = \frac{2T}{\pi} \log \frac{|d|^{1/2}T}{2\pi e {\rm{m}}(B)} + O(\log T).
\end{equation}

\begin{remark}
It is known to be difficult to prove the non-vanishing of central values of zeta or $L$-functions. 
We can regard Theorem \ref{th:m1} as an analogue or ``improvement'' of Theorem B since Theorem \ref{th:m1} implies that
$$
Q(1/2,a) > 0 \quad \mbox{ if and only if } \quad 0 < a < a_0 .
$$
In other words, there is the unique absolute threshold $a_0$ that determines `the central value' $Q(1/2,a)$ is positive, negative or zero. However, for Epstein zeta functions, there does not exist such an absolute threshold since we have no result for the sign of $\zeta_B(1/2)$ when $7.00554 < \kappa < 7.00556$. 
\end{remark}

\begin{remark}
The first statement of Corollary \ref{cor:lerch1} is the analogue of the approximation of $\zeta_B (s)$ near $s=1$ (see \cite[Theorem 1]{Sie}). 
Proposition \ref{pro:zero1} can be regarded as an analogue of theorems proved by Davenport \& Heilbronn \cite{DH} and Voronin \cite{Vo}. Proposition \ref{th:czaQAB1} is an analogue of (\ref{eq:RvMEp1}) 
\end{remark}

\begin{remark}
It should be emphasised that one has $N_{\rm{Ep}}^{\rm{CL}} (T) \gg T^{4/7-\varepsilon}$ but $N_{\rm{Q}}^{\rm{CL}} (T) \gg T$. 
The values of $\zeta_B (s)$ at positive integers are given in \cite[Theorem 1]{Sm}. On the other hand, it is proved in \cite[Theorem 1.2]{NaIn} that $\pi^{-2n} Q(2n,a)$, where $n \in {\mathbb{N}}$, can be expressed as a rational function with rational coefficients of $e^{2\pi ia}$. 
\end{remark}

The rest of this paper is organized as follows. Section 2.1 is devoted to the proofs of Theorem \ref{th:m1} and Corollary \ref{cor:1}. The proofs of Propositions \ref{pro:1} and \ref{pro:zero1} are given in Section 2.2. We prove Theorem \ref{th:lerch1} and Corollary \ref{cor:lerch1} in Section 2.3. We give proofs of Propositions \ref{pro:Hadpro1} and \ref{th:czaQAB1} in Section 2.4 and 2.5, respectively.

\section{Proofs}

\subsection{Proofs of Theorem \ref{th:m1} and Corollary \ref{cor:1}}
For $0 <a \le 1/2$, put
$$
Z(s,a) := \zeta (s,a) + \zeta (s,1-a), \qquad P(s,a) := {\rm{Li}}_s (e^{2\pi ia}) + {\rm{Li}}_s (e^{2\pi i(1-a)}).
$$
Note that one has $2Q(s,a) = Z(s,a) + P(s,a)$. By \cite[Lemma 4.1]{Na}, we have 
\begin{equation}\label{eq:feZQ1}
Z(1-s,a) = \Gamma_{\!\! \cos} (s) P(s,a), \qquad
P(1-s,a) = \Gamma_{\!\! \cos} (s) Z(s,a), 
\end{equation}
where $\Gamma_{\!\! \cos} (s)$ is defined in (\ref{eq:zfe1q}). 
In \cite[Lemma 4.4]{Na}, it is shown that
$$
\frac{\partial}{\partial a} Z (\sigma,a) < 0, \qquad \frac{\partial}{\partial a} P (\sigma,a) < 0, \qquad 0 < \sigma <1. 
$$
Hence we have the following inequality.
\begin{lemma}\label{lem:1}
Let $0<a<1/2$. Then it holds that
\begin{equation}\label{eq:Qin1}
\frac{\partial}{\partial a} Q (\sigma,a) < 0, \qquad 0 < \sigma <1. 
\end{equation}
\end{lemma}
From \cite[Theorem 1.3 and Proposition 4.5]{NaZC}, we have the following.
\begin{lemma}\label{lem:2}
One has\\
{\rm{(1)}} Let $1/6 \le a  \le 1/2$. Then one has $Z(\sigma,a) < 0$ for $0<\sigma <1$. \\
{\rm{(2)}} When $0 < a < 1/6$, the function $Z(\sigma,a)$ has precisely one simple zero in $(0, 1)$. Let $\beta_Z(a)$ denote the unique zero of $Z(\sigma,a)$ in $(0,1)$. Then the function $\beta_Z(a) \, \colon (0, 1/6) \to (0, 1)$ is a strictly decreasing $C^\infty$-diffeomorphism.
\end{lemma}

\begin{proof}[Proof of Theorem \ref{th:m1}]
From Lemma \ref{lem:2}, there is the unique $0 < a_0 < 1/6$ such that $Z(1/2, a_0)=0$. By the functional equations (\ref{eq:feZQ1}), we have $P(1/2, a_0)=0$. Hence, from (\ref{eq:Qin1}), the exists the unique $0 < a_0 <1/6$ such that
$$
2Q(1/2, a_0) = Z(1/2, a_0) + P(1/2, a_0) = 0.
$$
It should be noted that the $0 < a_0 <1/6$ above is given by (\ref{eq:av1}) numerically. By the functional equation (\ref{eq:zfe1q}), we have
$$
Q (1/2-\varepsilon,a_0) Q (1/2+\varepsilon,a_0) \ge 0
$$
if $\varepsilon >0$ is sufficiently small. The inequality above implies that
\begin{equation}\label{eq:dev0}
Q' (1/2, a_0) = 0, \qquad Q' (\sigma, a_0) := \frac{d}{d \sigma} Q(\sigma, a_0) .
\end{equation}
Hereafter, we only consider the open interval $(0,1/2)$ in virtue of (\ref{eq:zfe1q}). We have 
\begin{equation}\label{eq:Q-1/2}
2 Q(0,a) = Z (0,a) + P(0,a) = 0 -1 = -1, \qquad 0 < a \le 1/2
\end{equation}
from \cite[(4.11) and (4.12)]{Na}. Put
$$
a_\flat := 0.11837513961 < a_0 < 0.11837513962 =: a_\sharp .
$$
Then we have the following ten figures by Mathematica 11.3.

\begin{figure}[htbp]
\begin{center}
 \begin{tabular}{c}
 \begin{minipage}{0.5\hsize}
  \begin{center}
	\includegraphics[height=3.5cm, width=7cm]{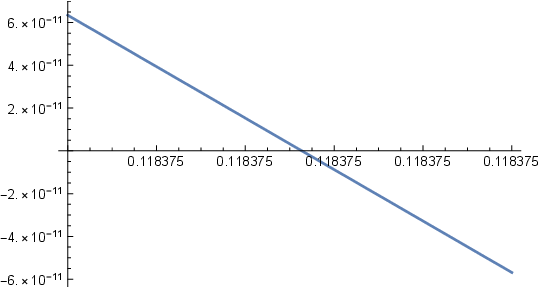}
	\caption{$\{ Z(1/2,a) : a_\flat \le a \le a_\sharp \}$}
	\label{f00}
  \end{center}
\end{minipage}
\begin{minipage}{0.5\hsize}
  \begin{center}
	\includegraphics[height=3.5cm, width=7cm]{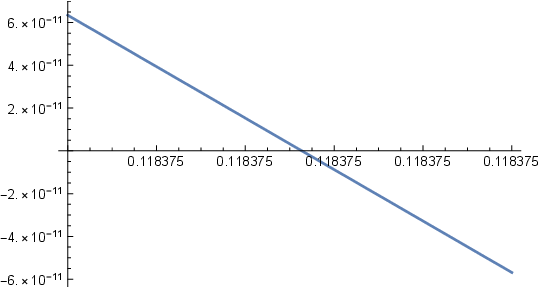}
	\caption{$\{ Q(1/2,a) : a_\flat \le a \le a_\sharp \}$}
	\label{f0}
  \end{center}
 \end{minipage}
 \end{tabular}
\end{center}
\end{figure}

\begin{figure}[htbp]
\begin{center}
 \begin{tabular}{c}
 \begin{minipage}{0.5\hsize}
  \begin{center}
	\includegraphics[height=3.5cm, width=7cm]{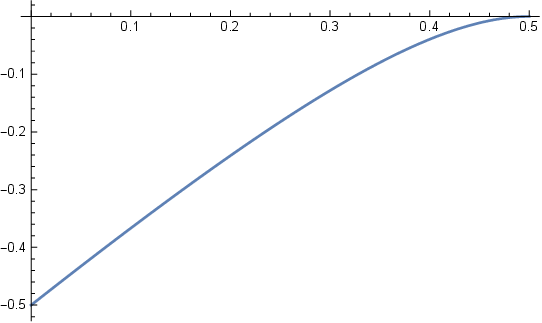}
	\caption{$\{ Q(\sigma, a_\flat) : 0 \le \sigma \le 1/2 \}$}
	\label{f1}
  \end{center}
\end{minipage}
\begin{minipage}{0.5\hsize}
  \begin{center}
	\includegraphics[height=3.5cm, width=7cm]{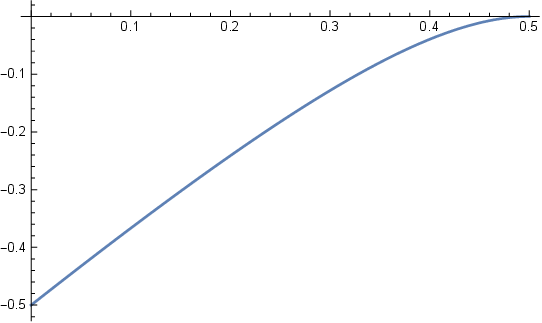}
	\caption{$\{ Q(\sigma, a_\sharp) : 0 \le \sigma \le 1/2 \}$}
	\label{f2}
  \end{center}
 \end{minipage}
 \end{tabular}
\end{center}

\end{figure}
\begin{figure}[htbp]
\begin{center}
 \begin{tabular}{c}
 \begin{minipage}{0.5\hsize}
  \begin{center}
	\includegraphics[height=3.5cm, width=7cm]{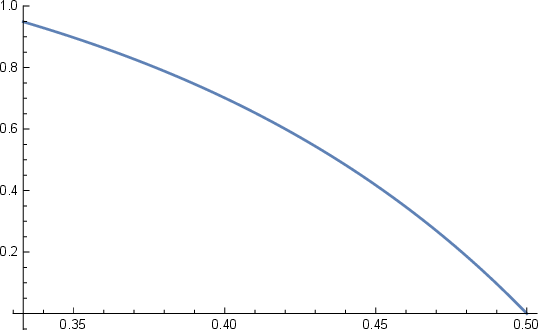}
	\caption{$\{ Q'(\sigma, a_\flat) : 1/3 \le \sigma \le 1/2 \}$}
	\label{f3}
  \end{center}
\end{minipage}
\begin{minipage}{0.5\hsize}
  \begin{center}
	\includegraphics[height=3.5cm, width=7cm]{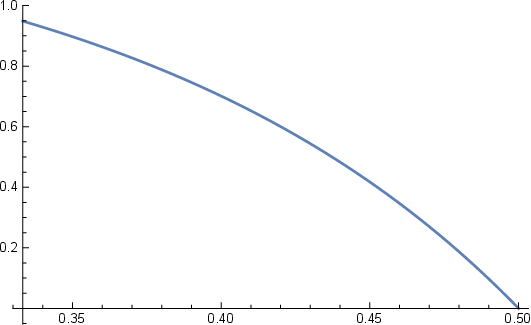}
	\caption{$\{ Q'(\sigma, a_\sharp) : 1/3 \le \sigma \le 1/2 \}$}
	\label{f3}
  \end{center}
 \end{minipage}
 \end{tabular}
\end{center}
\end{figure}

\begin{figure}[htbp]
\begin{center}
 \begin{tabular}{c}
 \begin{minipage}{0.5\hsize}
  \begin{center}
	\includegraphics[height=3.5cm, width=7cm]{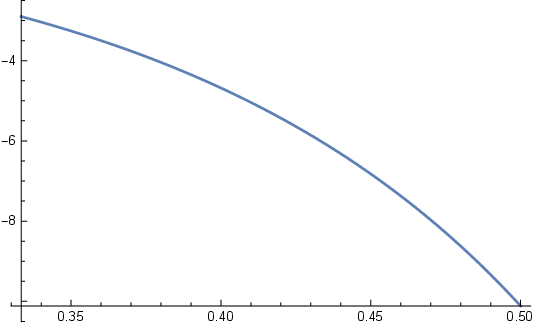}
	\caption{$\{ Q''(\sigma, a_\flat) : 1/3 \le \sigma \le 1/2 \}$}
	\label{f5}
  \end{center}
\end{minipage}
\begin{minipage}{0.5\hsize}
  \begin{center}
	\includegraphics[height=3.5cm, width=7cm]{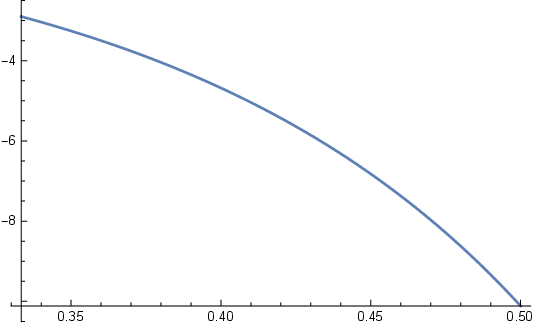}
	\caption{$\{ Q''(\sigma, a_\sharp) : 1/3 \le \sigma \le 1/2 \}$}
	\label{f6}
  \end{center}
 \end{minipage}
 \end{tabular}
\end{center}
\end{figure}

\begin{figure}[htbp]
\begin{center}
 \begin{tabular}{c}
 \begin{minipage}{0.5\hsize}
  \begin{center}
	\includegraphics[height=3.5cm, width=7cm]{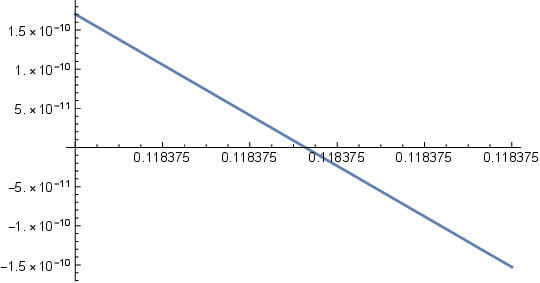}
	\caption{$\{ Q'(1/2,a) : a_\flat \le a \le a_\sharp \}$}
	\label{f7}
  \end{center}
\end{minipage}
\begin{minipage}{0.5\hsize}
  \begin{center}
	\includegraphics[height=3.5cm, width=7cm]{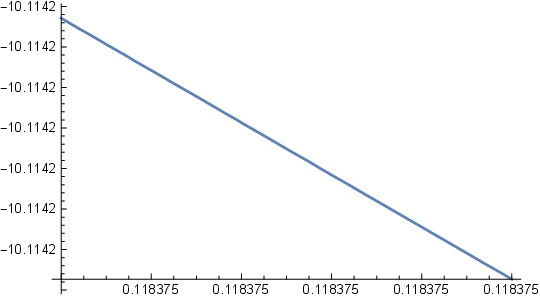}
	\caption{$\{ Q''(1/2,a) : a_\flat \le a \le a_\sharp \}$}
	\label{f8}
  \end{center}
 \end{minipage}
 \end{tabular}
\end{center}
\end{figure}
We can see that $Q'' (\sigma, a_0) < -2 < 0$ for all $1/3 \le \sigma \le 1/2$ by the seventh, eighth and tenth figures. Hence we have
\[
Q (\sigma, a_0) < 0, \qquad \sigma \in (1/3,1/2) .
\]
In addition, we can see that
\[
Q (\sigma, a_0) < 0, \qquad \sigma \in [0,1/3] 
\]
from (\ref{eq:Qin1}), $Q(0,a) =-1/2$ and the third and forth figures.
Therefore, by (\ref{eq:dev0}) and the functional equation (\ref{eq:zfe1q}), we have 
$$
Q(1/2, a_0) = Q'(1/2, a_0) = 0, \qquad Q(\sigma, a_0)<0, \quad \sigma \in (0,1/2) \cup (1/2,1)
$$
which imply the first statement of Theorem \ref{th:m1}. Moreover, from (\ref{eq:zfe1q}), (\ref{eq:Qin1}) and the first statement of this theorem, one has
\begin{equation}\label{ineq:Qneg1}
Q(\sigma,a)<0, \qquad \sigma \in (0,1), \quad a_0 < a \le 1/2.
\end{equation}
which implies the second statement of Theorem \ref{th:m1}. 

Next suppose $0 < a < a_0$. Then we have $Q(1/2,a) >0$ from (\ref{eq:Qin1}). On the other hand, one has $Q(0,a) = -1/2$ for all $0 < a \le 1/2$. Thus we have the third statement of Theorem \ref{th:m1} according to the intermediate value theorem. 
\end{proof}

\begin{proof}[Proof of Corollary \ref{cor:1}]
When $\sigma >1$ and $0<a \le 1/2$, we have
\begin{equation}\label{eq:posi1}
2Q(\sigma ,a) = Z(\sigma,a) + P(\sigma,a) > 
\sum_{n=0}^\infty \left( \frac{1}{(n+a)^\sigma} + \frac{1}{(n+1-a)^\sigma} - \frac{2}{(n+1)^\sigma} \right) >0.
\end{equation}
Moreover, we have $Q(\sigma,a)<0$ if $\sigma \in (0,1)$ and $a_0 < a \le 1/2$ from (\ref{ineq:Qneg1}). Recall that one has $Q(0,a)= -1/2$ by (\ref{eq:Q-1/2}) or \cite[(4.11) and (4.12)]{Na}. Hence, all real zeros of $Q(1-s,a)$ with $a_0 < a \le 1/2$ and $\sigma > 1$ come from $\cos (\pi s/2 ) = 0$ with $\sigma >1$ from (\ref{eq:zfe1q}) and the fact that $\Gamma (\sigma) >0$ when $0<\sigma <1$. Hence every real zero of $Q(s,a)$ with $a_0 < a \le 1/2$ and $\sigma <1$ is caused by $\cos ( \pi (1-s)/2) = 0$ with $\sigma < 0$, which is equivalent to that $s$ is a negative even integer.
\end{proof}

\subsection{Proofs of Propositions \ref{pro:1}{} and \ref{pro:zero1}}
From \cite[(4.1), (4.2), (4.3), (4.4), (4.5), (4.6), (4.7) and (4.8)]{Na}, it holds that
\begin{equation}\label{eq:Q1/2}
Z(s,1/2) = 2(2^s-1) \zeta (s), \qquad P(s,1/2) = 2(2^{1-s}-1) \zeta (s),
\end{equation}
\begin{equation}\label{eq:Q1/3}
Z(s,1/3) = (3^s-1) \zeta (s), \qquad P(s,1/3) = (3^{1-s}-1) \zeta (s),
\end{equation}
\begin{equation}\label{eq:Q1/4}
Z(s,1/4) = 2^s(2^s-1) \zeta (s), \qquad P(s,1/4) = 2^{1-s}(2^{1-s}-1) \zeta (s),
\end{equation}
\begin{equation}\label{eq:Q1/6}
Z(s,1/6) = (2^s-1)(3^s-1) \zeta (s), \qquad P(s,1/6) = (2^{1-s}-1)(3^{1-s}-1) \zeta (s).
\end{equation}

\begin{proof}[Proof of Proposition \ref{pro:1}]
When $a=1/2$, we have
$$
2Q(s,1/2) = 2 (2^s-1) \zeta (s) + 2 (2^{1-s}-1)\zeta (s) = 2(X-2+2X^{-1}) \zeta (s),
$$
where $0 \ne X:= 2^s$, form (\ref{eq:Q1/2}). The solutions of $X^2- 2X +2=0$ are $X=1\pm i$. Obviously, the all solutions of $2^s = 1\pm i$ are on the line $\sigma =1/2$.

Assume that $a=1/3$. By using (\ref{eq:Q1/3}), we have
$$
2Q(s,1/3) = (3^s-1) \zeta (s) + (3^{1-s}-1) \zeta (s) = (Y-2+3Y^{-1}) \zeta(s),
$$
where $0 \ne Y:=3^s$. We can easily see that $Y = 1\pm i \sqrt{2}$ are the roots of $Y^2-2Y+3=0$ and the all solutions of $3^s = 1\pm i \sqrt{2}$. are on the line $\sigma =1/2$.

Let $a=1/4$. By (\ref{eq:Q1/4}), it holds that
$$
2Q(s,1/4) =  2^s (2^s-1) \zeta (s) + 2^{1-s} (2^{1-s}-1) \zeta (s) =
(X^2 - X +4X^{-2} - 2X^{-1})\zeta (s).
$$
The solutions of $X^2 - X +4X^{-2} - 2X^{-1}=0$ are
$$
X = \frac{1}{4} \biggl( 1 + \sqrt{17} \pm i \sqrt{2 (7-\sqrt{17})} \biggr), \quad 
\frac{1}{4} \biggl( 1 - \sqrt{17} \pm i \sqrt{2 (7+\sqrt{17})} \biggr).
$$
The absolute value of the numbers above are $\sqrt{2}$. Therefore, the all roots of $X^2 - X +4X^{-2} - 2X^{-1}=0$, where $X:= 2^s$, are on the line $\sigma =1/2$.

Finally, suppose that $a=1/6$. From (\ref{eq:Q1/6}), one has
\begin{equation*}
\begin{split}
2Q(s,1/6) &= \bigl( (2^s-1)(3^s-1) + (2^{1-s}-1)(3^{1-s}-1) \bigr) \zeta (s) \\
& = (3^s-1) (2^{1-s}-1) \bigl( g_2^+(s) + g_3^+(1-s) \bigr) \zeta (s),
\end{split}
\end{equation*}
where the function $g_p^\pm (s)$ is defined as 
$$
g_p^\pm (s) := \frac{p^s \pm 1}{p^{1-s} \pm 1}, \qquad p=2,3.
$$
Obviously, we have $g_p^\pm(1-s) = 1/ g_p^\pm(s)$ from the definition. At the end of \cite[Section 3.1]{Na}, it is proved that $|g_2^+(s)|=1$ for $\sigma =1/2$, $|g_2^+(s)| <1$ if $\sigma >1/2$ and $|g_2^+(s)| >1$ if $\sigma <1/2$. By modifying the proof of this fact, we can prove that
$$
\bigl|g_p^- (s) \bigr|=1, \quad \sigma =1/2, \qquad  \bigl|g_p^- (s)\bigr|>1/2, \quad \sigma >1/2 , \qquad 
\bigl|g_p^- (s)\bigr|<1/2, \quad \sigma <1/2.
$$
Hence $(2^s-1)(3^s-1) + (2^{1-s}-1)(3^{1-s}-1)$ does not vanish when $\sigma \ne 1/2$. 
\end{proof}

Let $\varphi $ be the Euler totient function, $\chi$ be a primitive Dirichlet character of conductor of $q \in {\mathbb{N}}$ and $L(s,\chi)$ be the Dirichlet $L$-function. Let $G(r,\overline{\chi})$ denote the (generalized) Gauss sum $G(r,\overline{\chi}) := \sum_{n=1}^q \overline{\chi}(n)e^{2\pi irn/q}$ associated to a Dirichlet character $\overline{\chi}$. When $0< r/q \le 1/2$,  where $q$ and $r$ are relatively prime integers, we have 
\begin{equation}\label{eq:qq}
Q(s,r/q) = \frac{1}{2\varphi (q)} \sum_{\chi \!\!\! \mod q} 
\bigl(1+\chi(-1)\bigr) \bigl( \overline{\chi} (r) q^s + G(r,\overline{\chi}) \bigl) L(s,\chi)
\end{equation}
from \cite[(2.3)]{Na}. It is well-know that $\varphi (q) \le 2$ if and only if $q=1,2,3,4,6$. 

\begin{proof}[Proof of Proposition \ref{pro:zero1}]
The upper bound for the number of complex zeros  of $Q(s,a)$ is proved by the Bohr-Landau method (see for instance \cite[Theorem 9.15 (A)]{Tit}), the mean square of $\zeta (s,a)$ and ${\rm{Li}}_s (e^{2\pi ia})$ (see \cite[Theorem 4.2.1]{LauGa}) and the inequality
\begin{equation*}
\begin{split}
\int_2^T \bigl| Q(\sigma+it,a) \bigr|^2 dt & \le 
\int_2^T \bigl| \zeta(\sigma+it,a) \bigr|^2 dt + \int_2^T \bigl| \zeta(\sigma+it,1-a) \bigr|^2 dt \\
&+ \int_2^T \bigl| {\rm{Li}}_{\sigma+it} (e^{2\pi ia}) \bigr|^2 dt + \int_2^T \bigl| {\rm{Li}}_{\sigma+it} (e^{2\pi i(1-a)}) \bigr|^2 dt,
\qquad \sigma >1/2.
\end{split}
\end{equation*}
The lower bounds for the number of zeros of $Q(s,a)$ in the half-planes $1 < \sigma < 1+\delta$ and $-\delta < \sigma < 0$ are prove by (\ref{eq:zfe1q}), (\ref{eq:qq}) and \cite[Corollary]{SW}. Furthermore, the lower bounds for the number of zeros of $Q(s,a)$ in the half-planes $\sigma_1 < \sigma < \sigma_2$ and $1-\sigma_2 < \sigma < 1-\sigma_1$ are shown by (\ref{eq:zfe1q}), (\ref{eq:qq}) and \cite[Theorem 2]{KaczorowskiKulas} and the definition of $Q(s,a)$.
\end{proof}

\subsection{Proofs of Theorem \ref{th:lerch1} and Corollary \ref{cor:lerch1}}
Clearly, the functional equations in (\ref{eq:feZQ1}) imply
\begin{equation}\label{eq:GF1}
\Gamma_{\!\! \cos} (s)  \cdot \Gamma_{\!\! \cos} (1-s)  =1.
\end{equation}
Note that this equality is also proved by Euler's reflection formula
\[
\Gamma (s) \Gamma (1-s) = \frac{\pi}{\sin \pi s}, \qquad s \not \in {\mathbb{Z}}. 
\]

\begin{proof}[Proof of Theorem \ref{th:lerch1}]
From Lerch's formula (\ref{eq:lerch1}) and Euler's reflection formula, we have
\[
\Bigl[ \frac{d}{ds} Z (s,a) \Bigr]_{s=0} = \log \frac{\Gamma (a) \Gamma (1-a)}{2\pi} = - \log (\sin \pi a) - \log 2. 
\]
By (\ref{eq:feZQ1}) and (\ref{eq:GF1}), we can see that
\[
\Bigl[ \frac{d}{ds} P (s,a) \Bigr]_{s=0} = 
\Bigl[ \frac{d}{ds} \Bigl( \Gamma_{\!\! \cos} (1-s) Z(1-s,a) \Bigr) \Bigr]_{s=0} .
\]
In \cite[Theorem 1]{Bern}, Berndt show that
\begin{equation}\label{expa:1}
\zeta (s,a) = \frac{1}{s-1} + \sum_{n=0}^\infty \gamma_n (a) (s-1)^n = \sum_{n=-1}^\infty \gamma_n (a) (s-1)^n, 
\end{equation}
where $\gamma_{-1}(a) := 1$ and $\gamma_n(a)$ are given by
\[
\gamma_n(a) : = \frac{(-1)^n}{n!} \lim_{l \to \infty} \biggl( \sum_{k=0}^l \frac{\log^n (k+a)}{k+a} - \frac{\log^{n+1}(l+a)}{n+1} \biggr),
\qquad n \ge 0.
\]
Note that $\gamma_0(a) = - \psi (a)$, where $\psi (a)$ is the digamma function defined as the logarithmic derivative of the gamma function. We can easily show that
\[
\frac{d}{ds} \Bigl( \Gamma_{\!\! \cos} (1-s)  \Bigr) =  
2 \frac{d}{ds} \biggl( \frac{\Gamma (1-s)}{(2\pi)^{1-s}} \sin \Bigl( \frac{\pi s}{2} \Bigr) \biggr) = D_1 (s) + D_2 (s) + D_3 (s),
\]
where the functions $D_1 (s)$, $D_2 (s)$ and $D_3 (s)$ are defined as
\begin{equation*}
\begin{split}
&D_1 (s) := 2 (2\pi)^{s-1} \Gamma (1-s) \sin \Bigl( \frac{\pi s}{2} \Bigr) \log 2\pi,\\
&D_2 (s) := -2(2\pi)^{s-1} \sin \Bigl( \frac{\pi s}{2} \Bigr) \Gamma (1-s) \psi (1-s),\\
&D_3 (s) := \pi (2\pi)^{s-1} \Gamma (1-s) \cos \Bigl( \frac{\pi s}{2} \Bigr).
\end{split}
\end{equation*}
From the expansion (\ref{expa:1}), we have
\[
Z(1-s,a) = \sum_{n=-1}^\infty \delta_n (a) (-s)^n, \qquad \frac{d}{ds} Z(1-s,a) = \sum_{n=-1}^\infty n \delta_n (a) (-s)^{n-1},
\]
where $\delta_n (a) := \gamma_n (a) + \gamma_n (1-a)$. Hence one has
\[
\bigl[ D_1 (s) Z(1-s,a) \bigr]_{s=0}= \frac{\log 2\pi}{\pi} \Bigl[ \sin \Bigl( \frac{\pi s}{2} \Bigr) Z(1-s,a) \Bigr]_{s=0} = 
\frac{\log 2 \pi}{\pi} \Bigl[ \frac{\pi s}{2} \cdot \frac{2}{-s} \Bigr]_{s=0} = - \log 2\pi , 
\]
\[
\bigl[ D_2 (s) Z(1-s,a) \bigr]_{s=0} = - \frac{\psi (1)}{\pi} \Bigl[ \sin \Bigl( \frac{\pi s}{2} \Bigr) Z(1-s,a) \Bigr]_{s=0} = \psi (1) = -\gamma_E ,
\]
\begin{equation*}
\begin{split}
&\Bigl[ D_3 (s) Z(1-s,a) + \Gamma_{\!\! \cos} (1-s) \frac{d}{ds} Z(1-s,a) \Bigr]_{s=0} \\ &= 
\biggl[ \frac{1}{2} \sum_{n=-1}^\infty \delta_n (a) (-s)^n - \frac{1}{\pi}
\sin \Bigl( \frac{\pi s}{2} \Bigr) \sum_{n=-1}^\infty n \delta_n (a) (-s)^{n-1} \biggr]_{s=0} = \frac{\delta_0 (a)}{2} .
\end{split}
\end{equation*}
Thus we obtain (\ref{eq:thmlerch1}) by $\gamma_0(a) = - \psi (a)$ and $\delta_n (a) := \gamma_n (a) + \gamma_n (1-a)$.
\end{proof}

\begin{proof}[Proof of Corollary \ref{cor:lerch1}]
According to (\ref{expa:1}), we can easily see that
\[
\biggl[ Z(s,a) -\frac{2}{s-1} \biggr]_{s=1} = \delta_0 (a) = - \psi (a) -\psi(1-a) .
\]
On the other hand, it is proved 
\[
P(1,a) = - 2\log (2 \sin \pi a) = - 2\log (\sin \pi a) - 2\log 2 
\]
in \cite[(4.12)]{NaZC}. Therefore, we obtain
\begin{equation}\label{eq:kro2}
\biggl[ Q(s,a) -\frac{1}{s-1} \biggr]_{s=1} = 2 Q' (0,a) + \gamma_E + \log 2 \pi .
\end{equation}
We can find that
\[
\biggl[ \zeta (s) -\frac{1}{s-1} \biggr]_{s=1} = \gamma_E \quad \mbox{ and } \quad \zeta '(0) = - \frac{\log 2\pi}{2}
\]
in \cite[(2.1.16)]{Tit} and \cite[(2.4.5)]{Tit}, respectively. Thus we have (\ref{eq:kro1}). In addition, one has
\[
\psi (a) = \frac{\Gamma' (1)}{\Gamma (1)} + \sum_{n=1}^\infty \biggl( \frac{1}{n+1} - \frac{1}{n+a} \biggr), \qquad
\psi (a) = - \frac{1}{a} + \psi (1+a)
\]
(see \cite[(1.2.13) and (1.2.15)]{AAR}). Thus we have
\[
\psi (1+a), \psi(1-a) = O (1), \qquad 0 \le a \le 1/2.
\]
Hence we obtain (\ref{eq:corasym1}). 
It is shown that
\[
\psi (r/q) + \psi (1-r/q) = -2\gamma_E - 2\log q + 2 \sum_{n=1}^{q-1} \cos \frac{2\pi rn}{q} \log \Bigl( 2 \sin \frac{\pi n}{q} \Bigr)
\]
in \cite[(1.2.17)]{AAR}. The equation above and Theorem \ref{th:lerch1} imply (\ref{eq:corlerch1}). 
\end{proof}

\subsection{Proof of Proposition \ref{pro:Hadpro1}}
First we show the following zero-free region of the quadrilateral zeta function.
\begin{lemma}\label{lem:nonzero1}
For any $0< a \le 1/2$ and $\eta >0$, there exists $\sigma_a' (\eta) \ge 3/2$ such that $|Q(s,a)| > \eta$ for all $\Re (s) \ge \sigma_a' (\eta)$. 
\end{lemma}
\begin{proof}
Fix $0< a < 1/2$. When $\sigma \ge 3/2$ we have
\begin{equation*}
\begin{split}
2|Q(s,a)| & \ge a^{-\sigma} - \sum_{n=1}^\infty \frac{1}{(n+a)^\sigma} - \sum_{n=0}^\infty \frac{1}{(n+1-a)^\sigma} -
2 \sum_{n=1}^\infty \frac{|\cos 2 \pi na|}{n^\sigma} \\ & \ge 
a^{-\sigma} - (1-a)^{-\sigma} - \zeta (\sigma,1+a) - \zeta (\sigma, 2-a) - 2\zeta (\sigma) \\ & \ge 
a^{-\sigma} - (1-a)^{-\sigma} - \zeta (\sigma) - \zeta (\sigma) - 2\zeta (\sigma) \\ & \ge 
a^{-\sigma} - (1-a)^{-\sigma} - 4 \zeta (3/2) 
\end{split}
\end{equation*}
by the series expression of $Q(s,a)$. Note that the function $a^{-\sigma} - (1-a)^{-\sigma}$ is monotonically increasing with respect to $\sigma >1$. Hence the inequality above implies Lemma \ref{lem:nonzero1} with $0< a < 1/2$. When $a=1/2$ and $\sigma \ge 3/2$, one has
\begin{equation*}
\begin{split}
&2|Q(s,1/2)| \ge 2 (1/2)^{-\sigma} - 2 \sum_{n=1}^\infty \frac{1}{(n+1/2)^\sigma} -
2 \sum_{n=1}^\infty \frac{|\cos \pi n|}{n^\sigma} \\ & \ge 
2^{\sigma+1} - 2 \zeta (\sigma,3/2) - 2\zeta (\sigma) \ge 
2^{\sigma+1} - 4 \zeta (3/2) .
\end{split}
\end{equation*}
Hence we have this lemma for $0< a \le 1/2$. 
\end{proof}

\begin{lemma}\label{lem:order1}
Let $0 <a \le 1/2$. Then $\xi_Q(s,a)$ is an entire function of order $1$.
\end{lemma}
\begin{proof}
By the functional equation (\ref{eq:zfe1q}), it suffices to estimate $|\xi_Q(s,a)|$ on the half-plane $\Re (s) \ge 1/2$. From the approximate functional equations of $\zeta (s,a)$ and ${\rm{Li}}_s (e^{2\pi ia})$ (see \cite[Theorems 3.1.2 and 3.1.3]{LauGa}), it holds that
$$
\bigl| s(s-1) Q(s,a) \bigr| = O_a \bigl( |s|^3 \bigr). 
$$
According to the inequalities
$$
\bigl|\Gamma (s) \bigr| \le \exp (c_1|s| \log |s|), \qquad  \bigl| \pi^{-s/2} \bigr| \le \exp (c_2|s|),
$$
where $c_1$ and $c_2$ are some positive constants (see for instance \cite[p.~20]{KV}), $\xi_Q(s,a)$ is a function of at most $1$. The order is exactly $1$ since one has 
$$
4 Q(\sigma,a ) \ge a^{-\sigma}, \qquad \log \Gamma (\sigma) \ge (\sigma - 1/2) \log \sigma - 2 \sigma,
$$
where $\sigma >0$ is sufficiently large, by the general Dirichlet series expression of $Q(s,a)$ and Stirling's formula.
\end{proof}

\begin{proof}[Proof of Proposition \ref{pro:Hadpro1}]
By Lemma \ref{lem:order1} and Hadamard's factorization theorem, we only determine the constants $A$ and $B(a)$. From (\ref{eq:Q-1/2}), one has
$$
\xi_Q (0,a) = - Q(0,a) = 1/2 = e^A. 
$$
Hence it holds that
$$
Q(s,a) = \frac{e^{B(a)s} \pi^{s/2}}{s(s-1) \Gamma (s/2)} \prod_{\rho_a} \biggl( 1 - \frac{s}{\rho_a} \biggr) e^{s/\rho_a} = 
\frac{e^{(B(a)+(\log \pi)/2)s}}{2(s-1) \Gamma (s/2+1)} \prod_{\rho_a} \biggl( 1 - \frac{s}{\rho_a} \biggr) e^{s/\rho_a}. 
$$
By taking the logarithmic derivative of the formula above, we obtain
\begin{equation}\label{eq:logdev1}
\frac{Q' (s,a)}{Q (s,a)} = B(a) + \frac{\log \pi}{2} - \frac{1}{s-1} - \frac{1}{2} \frac{\Gamma' (s/2+1)}{\Gamma (s/2+1)}
+ \sum_{\rho_a} \biggl( \frac{1}{s-\rho_a} + \frac{1}{\rho_a} \biggr)
\end{equation}
By making $s \to 0$, we have
$$
\frac{Q' (0,a)}{Q (0,a)} = B(a) + 1 + \frac{\gamma_E+\log \pi}{2}
$$
which implies Proposition \ref{pro:Hadpro1}.
\end{proof}

\subsection{Proof of Proposition \ref{th:czaQAB1}}
By using the Hadamard product formula for $Q(s,a)$, we show the following approximate formula for $Q' (s,a) /Q(s,a)$ in terms of zeros near to $s$ which is an analogue of \cite[Corollary 1.6.3]{KV} or \cite[Theorem 9.6]{Tit}. 
\begin{proposition}\label{cor:corHadpro1}
Let $0 < a \le 1/2$, $\sigma_a := \sigma_a'+1$, where $\sigma_a' := \sigma_a' (\eta) \ge 3 /2$ is given in Lemma \ref{lem:nonzero1}, and $\gamma_a$ be the imaginary part of the zeros of $\xi_Q (s,a)$. Then for $1- \sigma_a \le \sigma \le \sigma_a$ and $s=\sigma+it$, it holds that
\begin{equation}\label{eq:corhadpro1}
\frac{Q' (s,a)}{Q (s,a)} = - \frac{1}{s-1} + \sum_{|t-\gamma_a|  \le 1} \frac{1}{s-\rho_a} + O_a \bigl( \log(|t|+2) \bigr). 
\end{equation}
\end{proposition}
\begin{proof}
From (\ref{eq:logdev1}) and the Stirling formula
\begin{equation}\label{eq:star1}
\frac{\Gamma' (s/2)}{\Gamma (s/2)} = \log (s/2) + O_a\bigl( |s|^{-1} \bigr),
\end{equation}
it holds that
\begin{equation}\label{eq:4.13}
\frac{Q' (s,a)}{Q (s,a)} = - \frac{1}{s-1} + \sum_{\rho_a} \biggl( \frac{1}{s-\rho_a} + \frac{1}{\rho_a} \biggr) + O_a \bigl( \log(|t|+2) \bigr).
\end{equation}
By putting $s=\sigma_a+it$, we obtain
$$
\sum_{\rho_a} \biggl( \frac{1}{\sigma_a+it-\rho_a} + \frac{1}{\rho_a} \biggr) = O_a \bigl( \log(|t|+2) \bigr)
$$
since the general Dirichlet series of $Q(s,a)$ converges absolutely and does not vanish when $\sigma \ge \sigma_a = \sigma_a'+1$ from Lemma \ref{lem:nonzero1}. Let $\rho_a := \beta_a + i \gamma_a$. Then, the real part of the each term in the infinite summation above can be expressed as
$$
\frac{\sigma_a - \beta_a}{(\sigma_a - \beta_a)^2 + (t - \gamma_a)^2} + \frac{\beta_a}{\beta_a^2 + \gamma_a^2}
\ge \frac{\sigma_a - \beta_a}{(\sigma_a - \beta_a)^2 + (t - \gamma_a)^2} \gg \frac{1}{1 + (t - \gamma_a)^2}.
$$
Therefore, we have
\begin{equation}\label{eq:5.19}
\sum_{\rho_a} \frac{1}{1 + (t - \gamma_a)^2} \ll_a \log(|t|+2).
\end{equation}
We have the left hand side of (\ref{eq:5.19}) $\gg N( T+1, Q(s,a) ) - N( T, Q(s,a) )$ since one has $(1 + (t - \gamma_a)^2)^{-1} \gg 1$ when $0 < t \le \gamma_a \le t+1$. Thus we obtain
\begin{equation}\label{eq:5.20}
N\bigl( T+1, Q(s,a) \bigr) - N\bigl( T, Q(s,a) \bigr) \ll_a \log(|t|+2).
\end{equation}
From (\ref{eq:4.13}) and we subtract the same equality with $s=\sigma_a+it$,
\begin{equation}\label{eq:5.21}
\frac{Q' (s,a)}{Q (s,a)} = \sum_{\rho_a} \biggl( \frac{1}{s-\rho_a} - \frac{1}{\sigma_a+it - \rho_a} \biggr) + O_a \bigl( \log(|t|+2) \bigr).
\end{equation}
According to the assumption $1-\sigma_a \le \sigma \le \sigma_a$, one has
$$
\biggl| \frac{1}{s-\rho_a} - \frac{1}{\sigma_a+it - \rho_a} \biggr| =
\frac{\sigma_a - \sigma}{|(s-\rho_a)(\sigma_a+it - \rho_a)|} \le \frac{2\sigma_a}{|t - \gamma_a|^2} .
$$
Hence the the infinite summation in (\ref{eq:5.21}) with $|t-\gamma_a| > 1$ is $O_a ( \log(|t|+2) )$ by (\ref{eq:5.19}). Moreover, it holds that
$$
\sum_{|t-\gamma_a| \le 1} \frac{1}{\sigma_a+it - \rho_a} = O_a \bigl( \log(|t|+2) \bigr)
$$
from (\ref{eq:5.20}) and the inequality $|\sigma_a+it - \rho_a| \ge 1$ which is proved by Lemma \ref{lem:nonzero1} and the definitions of $\sigma_a'$ and $\sigma_a$. Therefore, we have (\ref{eq:corhadpro1}) by (\ref{eq:5.21}). 
\end{proof}

\begin{proof}[Proof of Proposition \ref{th:czaQAB1}]
The proof method is using the argument principle (see for instance \cite[Section 1.8]{KV} and \cite[Section 9.3]{Tit}).
Fix $0< a \le 1/2$ and assume that $Q(s,a)$ does not vanish on the line $\Im (s) =T$. Let $\mathcal{C}$ be the rectangular contour with vertices $s =\sigma_a \pm T$, $s= 1-\sigma_a \pm T$, where $\sigma_a>0$ is given in Proposition \ref{cor:corHadpro1}. By Theorem \ref{th:m1} and Littlewood's Lemma, it holds that
\begin{equation}\label{eq:little1}
N\bigl( T, Q(s,a) \bigr) = \frac{1}{2\pi i} \int_{\mathcal{C}} \frac{\xi_Q' (s,a)}{\xi_Q (s,a)} ds + O_a (1). 
\end{equation}
From the definition of $\xi_Q (s,a)$, one has
\begin{equation}\label{eq:KV3}
\frac{\xi_Q' (s,a)}{\xi_Q (s,a)} = \frac{1}{s} + \frac{1}{s-1} - \frac{\log \pi}{2} 
+ \frac{1}{2} \frac{\Gamma' (s/2)}{\Gamma (s/2)} + \frac{Q' (s,a)}{Q (s,a)}. 
\end{equation}
The equation (\ref{eq:little1}) can be expressed as
\begin{equation}\label{eq:little2}
\begin{split}
N\bigl( T, Q(s,a) \bigr) + O_a (1) &= \frac{1}{2\pi} \int_{-T}^T 
\biggl( \frac{\xi_Q' (\sigma_a+it,a)}{\xi_Q (\sigma_a+it,a)} - \frac{\xi_Q' (1-\sigma_a+it,a)}{\xi_Q (1-\sigma_a+it,a)} \biggr) dt \\
&\quad + \frac{1}{2\pi i} \int_{1-\sigma_a}^{\sigma_a}
\biggl( \frac{\xi_Q' (\sigma-iT,a)}{\xi_Q (\sigma-iT,a)} - \frac{\xi_Q' (\sigma+iT,a)}{\xi_Q (\sigma+iT,a)} \biggr) d\sigma \\
&=: I_1+I_2,
\end{split}
\end{equation}
where $I_1$ and $I_2$ are the first and second integrals in the last formula. 

First we find an upper bound of for $|I_2|$. By (\ref{eq:star1}) and the definition of $I_2$, we have
$$
I_2 = \frac{1}{2\pi i} \int_{1-\sigma_a}^{\sigma_a} 
\biggl( \frac{Q' (\sigma-iT,a)}{Q (\sigma-iT,a)} - \frac{Q' (\sigma+iT,a)}{Q (\sigma+iT,a)} \biggr) d\sigma + O_a(\log T). 
$$
From Proposition \ref{cor:corHadpro1} one has 
\begin{equation}\label{eq:KV6}
\frac{Q' (s,a)}{Q (s,a)} = \sum_{|T-\gamma_a| \le 1} \frac{1}{\sigma+iT-\rho_a} + O_a(\log T),
\end{equation}
where $\rho_a$ are the zeros of $\xi_Q (s,a)$. Now let ${\mathcal{C}}'$ be the rectangular contour with vertices $s=1-\sigma_a+iT$, $s=\sigma_a+iT$, $s=\sigma_a+i(T-2)$, $s=1-\sigma_a+i(T-2)$. Note that the number of $\rho_a$ satisfying $T-2 \le \Im (\rho_a) \le T$ is $O_a(\log T)$ from (\ref{eq:5.20}). Hence, we obtain 
$$
\int_{{\mathcal{C}}'} \left( \sum_{|T-\gamma_a| \le 1} \frac{1}{s-\rho_a} \right) ds = O_a(\log T).
$$
This integral over ${\mathcal{C}}'$ can also be written as 
\begin{equation*}
\begin{split}
&\int_{{\mathcal{C}}'} \left( \sum_{|t-\gamma_a| \le 1} \frac{1}{s-\rho_a} \right) ds = 
\sum_{T-1 \le \gamma_a \le T+1} \int_{{\mathcal{C}}'} \frac{ds}{s-\rho_a} = \\
&- \sum_{T-1 \le \gamma_a \le T+1} \int_{1-\sigma_a}^{\sigma_a} \frac{d\sigma}{\sigma+iT-\rho_a}
+ i \sum_{T-1 \le \gamma_a \le T+1} \int_{T-2}^T \frac{dt}{\sigma_a+iT-\rho_a} \\
&- i \sum_{T-1 \le \gamma_a \le T+1} \int_{T-2}^T \frac{dt}{1-\sigma_a+iT-\rho_a} + 
\sum_{T-1 \le \gamma_a \le T+1} \int_{1-\sigma_a}^{\sigma_a} \frac{d\sigma}{\sigma+i(T-2)-\rho_a}
\end{split}
\end{equation*}
We can easily see that the last three sums are $O_a(\log T)$. Hence the first sum is also $O_a(\log T)$. From this fact and (\ref{eq:KV6}), we can conclude that
$$
I_2 = O_a(\log T). 
$$

Next we estimate $I_1$, making use of (\ref{eq:KV3}) and the relation
$$
\frac{\xi_Q' (s,a)}{\xi_Q (s,a)} = - \frac{\xi_Q' (1-s,a)}{\xi_Q (1-s,a)}.
$$
From (\ref{eq:star1}), (\ref{eq:KV3}) and the formula above, we have
\begin{equation*}
\begin{split}
&\frac{\xi_Q' (\sigma_a+it,a)}{\xi_Q (\sigma_a+it,a)} - \frac{\xi_Q' (1-\sigma_a+it,a)}{\xi_Q (1-\sigma_a+it,a)} =
\frac{\xi_Q' (\sigma_a+it,a)}{\xi_Q (\sigma_a+it,a)} + \frac{\xi_Q' (\sigma_a-it,a)}{\xi_Q (\sigma_a-it,a)}  \\
=\, & \frac{\sigma_a^2}{\sigma_a^2+t^2} + \frac{(1-\sigma_a)^2}{(1-\sigma_a)^2+t^2} - \log \pi + 
\frac{1}{2} \log \frac{\sigma_a^2+t^2}{\sigma_a^2} + O_a\bigl( (\sigma_a^2+t^2)^{-1/2} \bigr) \\
& + \frac{Q' (\sigma_a+it,a)}{Q (\sigma_a+it,a)} + \frac{Q' (\sigma_a-it,a)}{Q (\sigma_a-it,a)} .
\end{split}
\end{equation*}
Obviously, one has
\begin{equation*}
\begin{split}
\int_{-T}^T \frac{Q' (\sigma_a+it,a)}{Q (\sigma_a+it,a)} dt = 
\int_{-T}^T \frac{(a^{-\sigma_a-it})'}{ a^{-\sigma_a-it}} dt + 
\int_{-T}^T \frac{(a^{\sigma_a+it} Q (\sigma_a+it,a))'}{ a^{\sigma_a+it} Q (\sigma_a+it,a)} dt.
\end{split}
\end{equation*}
For the first integral, we have
$$
\int_{-T}^T \frac{(a^{-\sigma_a-it})'}{ a^{-\sigma_a-it}} dt = -i \Bigl[ \log a^{-\sigma_a-it} \Bigr]_{-T}^T = -2T \log a. 
$$
For some $\theta >0$, without loss of generality we can assume that $\sigma_a$ satisfies
$$
2 \Re \bigl( a^{\sigma_a+it} Q (\sigma_a+it,a) \bigl) \ge 1 - \Bigl( \frac{a}{1-a} \Bigr)^{\sigma_a} - 4a^{\sigma_a} \zeta (\sigma_a) > \theta
$$
when $0 < a <1/2$ by modifying the proof of Lemma \ref{lem:nonzero1}. One can assume similarly when $a=1/2$ (see the proof of Lemma \ref{lem:nonzero1}). Hence we obtain 
$$
\int_{-T}^T \frac{(a^{\sigma_a+it} Q (\sigma_a+it,a))'}{ a^{\sigma_a+it} Q (\sigma_a+it,a)} dt = 
-i \Bigl[ \bigl( \log a^{\sigma_a+it} Q (\sigma_a+it,a) \bigr) \Bigr]_{-T}^T = O_a (1). 
$$
Therefore, it holds that
\begin{equation*}
\begin{split}
I_1 &= \frac{1}{2\pi} \int_{-T}^T
\biggl( \frac{\xi_Q' (\sigma_a+it,a)}{\xi_Q (\sigma_a+it,a)} + \frac{\xi_Q' (\sigma_a-it,a)}{\xi_Q (\sigma_a-it,a)} \biggr) dt \\
&= -\frac{T}{\pi} \log \pi - \frac{T}{\pi} \log 2 + \frac{T}{\pi} \log T - \frac{T}{\pi} - \frac{2T}{\pi} \log a + O_a(\log T) \\ 
&= \frac{T}{\pi} \log T - \frac{T}{\pi} \log (2 e \pi a^2) + O_a(\log T) .
\end{split}
\end{equation*}
The theorem in this case follows from the formula above and the bound for $I_2$. If we suppose that $Q(s,a)$ has zeros on the line $\Im (s) =T$, then the theorem follows from the case above of the theorem along with (\ref{eq:5.20}).
\end{proof}

\subsection*{Acknowledgments}
The author was partially supported by JSPS grant 16K05077. The author would like to thank Mr.~Hiroki Sakurai for his helpful comments for Proposition \ref{th:czaQAB1} and its proof.

 
\end{document}